\documentclass[reqno,oneside]{amsart}

\usepackage{hyperref}
\usepackage{amssymb,amsmath,amssymb,amsfonts,amsthm}
\usepackage{cases}
\usepackage{color}
\usepackage{empheq}
\usepackage{url}
\usepackage{geometry}
\usepackage{cite,mathtools}

\numberwithin{equation}{section}

\theoremstyle{plain}
\newtheorem{theorem}{Theorem}[section]
\newtheorem{lemma}{Lemma}[section]
\newtheorem{proposition}{Proposition}[section]

\theoremstyle{definition}

\newtheorem{remark}{Remark}[section]

\allowdisplaybreaks[4]

\def\BR{\mathbb R}

\def\rd{\mathrm d}

\def\e{\mathrm e}

\def\La{\Lambda}
\def\Om{\Omega}
\def\al{\alpha}
\def\be{\beta}
\def\ga{\gamma}
\def\de{\delta}

\def\ve{\varepsilon}
\def\ze{\zeta}
\def\te{\theta}

\def\la{\lambda}
\def\si{\sigma}

\def\f{\frac}
\def\nb{\nabla}
\def\ov{\overline}
\def\pa{\partial}
\def\tri{\triangle}
\def\wt{\widetilde}

\title[Energy decay and blow-up of viscoelastic wave equations]{Energy decay and blow-up of viscoelastic wave equations with polynomial nonlinearity and damping}
\author[Qingqing Peng]{Qingqing Peng$^{1,2}$}
\thanks{$^1$School of Mathematics and Statistics \& Hubei Key Laboratory of Engineering Modeling and Scientific Computing, Huazhong University of Science and Technology, Wuhan 430074, China.
}
\thanks{$^2$Department of Mathematics, Kyoto University, Kitashirakawa-Oiwakecho, Sakyo-ku, Kyoto 606-8502, Japan.}
\author[Yikan Liu]{Yikan Liu$^{2,*}$}
\thanks{$^*$Corresponding author.  E-mail: {\tt liu.yikan.8z@kyoto-u.ac.jp}}

\begin{document}

\maketitle

\begin{abstract}
This paper is concerned with the energy decay and the finite time blow-up of the solution to a viscoelastic wave equation with polynomial nonlinearity and weak damping. We establish explicit and general decay results for the solutions by imposing polynomial conditions on the relaxation function, provided that the initial energy is sufficiently small. Furthermore, we derive an upper bound for the blow-up time when the initial energy is less than the depth of the potential well by utilizing Levine's convexity method. Additionally, we provide a lower bound for the blow-up time if the solution blows up.
\end{abstract}


\section{Introduction}

Let $\BR_+:=(0,\infty)$, $\overline{\BR_+}:=[0,\infty)$ and $\Om\subset\BR^n$ ($n\ge3$) be a bounded and connected domain with a smooth boundary $\pa\Om$. The main purpose of this work is to study the energy decay rate and blow-up of the following initial-boundary value problem for a viscoelastic wave equation with polynomial nonlinearity and weak damping
\begin{equation}\label{1.1}
\left\{\begin{alignedat}{2}
& u_{tt}-\tri u+\int_0^t f(t-s)\tri u(s)\,\rd s+a(x)u_t=k(x)|u|^{p-2}u & \quad &\mbox{in }\Om\times\BR_+,\\
& u=u^0,\ u_t=u^1 & \quad &\mbox{in }\Om\times\{0\},\\
& u=0 & \quad &\mbox{on }\pa\Om\times\BR_+.
\end{alignedat}\right.
\end{equation}
Here $u^0\in H_0^1(\Om)$, $u^1\in L^2(\Om)$ and $2<p<\f{2n-2}{n-2}$. The coefficient $a(x)$ of the damping term is given as
$$
a(x)=|x|^{-\si},\quad0\le\si\le2.
$$
Moreover, the space-dependent coefficient $k(x)$ satisfies
$$
k\in C^1(\ov\Om),\quad k\ge0\mbox{ on }\ov\Om.
$$

Recent decades have witnessed significant progress regarding well-posedness and energy decay in wave equations with either weak or strong damping. Webb \cite{W1} was among the first to investigate the well-posedness of strong solutions through operator theory while discussing long-time behavior using Lyapunov stability techniques applied to semilinear wave equations with strong damping and external forcing. In particular, Gazzola and Squassina \cite{G1} examined semilinear strongly damped wave equations incorporating frictional damping. They established both existence and nonexistence results for solutions when the initial energy is less than the depth potential well. Meanwhile, they also shown blow-up results by employing potential well methods along with Levine's concavity technique. Subsequently, Ma and Fang \cite{M1} demonstrated both existence and nonexistence of global weak solutions for the wave equations. They derived energy decay estimates via potential well methods combined with logarithmic Sobolev inequalities. Lian and Xu \cite{L1} further explored the global existence, energy decay and infinite-time blow-up using methodologies akin to those presented in \cite{M1}. Di, Shang and Song \cite{D1} investigated semilinear wave equations characterized by logarithmic nonlinearity coupled with strong damping. They proved global well-posedness along with polynomial or exponential stability through potential methods paired with Lyapunov skills while also considering various blow-up scenarios.

Taking memory term into account, several results concerning decay and blow-up of solutions have been established under certain assumptions regarding the kernel function $f$. Messaoudi \cite{M2} demonstrated a general decay result for the viscoelastic wave equation,
which is not necessarily of exponential or polynomial types. Subsequently, Belhannache et al. \cite{B1} examined the asymptotic stability for a viscoelastic equation with nonlinear damping and very general relaxation functions. They proved the polynomial stability of the system by employing the multiplier method alongside properties of convex functions, provided that the kernel function satisfies
$$
g'(t)\le-\xi(t)G(g(t)),
$$
where $\xi(t)$ is a non-increasing function and $G$ is a strictly increasing and strictly convex $C^2$ function. Al-Gharabli \cite{A1} established stability result for a viscoelastic plate equation with logarithmic sources, where the conditions on the kernel function are similar to those in \cite{B1}. Messaoudi and Al-Khulaifi \cite{M3} investigated energy decay rates for viscoelastic wave equations with weak damping, where the kernel function satisfies
$$
g'(t)\le-\xi(t)g^q(t),\quad1\le q<\f32
$$
and $\xi(t)$ is a non-increasing function. Later, Al-Mahdi and Al-Gharabli \cite{A2} discussed the energy decay in infinite memory wave equations featuring nonlinear damping. Remarkably, Li and Gao \cite{F1} proved blow-up results for viscoelastic wave problems using Lyapunov methods and further determined the blow-up times. Ha and Park \cite{H1} also established blow-up results along with the local existence of solutions for viscoelastic wave equations incorporating logarithmic sources through potential well methods combined with Lyapunov techniques. Recently, the authors investigated similar decay properties for nonlocal viscoelastic equations with damping and nonlinearity in \cite{PL25,PL26}.

However, to our knowledge, there has been limited research focusing on models characterized by singular space-dependent coefficients of frictional damping $a(x)u_t$ and space-dependent coefficient of source $k(x)|u|^{p-2}u$. Most recently, there are some new advances on the study of hyperbolic equations without strong damping but with weighted weakly damping terms. For example, Pata and Zelik \cite{P1} proved that the asymptotic profile of the solution by a solution of the corresponding heat equation in the $L^2$ sense for the Cauchy problem of the wave equation with space-dependent coefficient of weak damping as
$$
u_{tt}-\tri u+V(x) u_t=0\quad\mbox{in }\BR^n\times\BR_+,
$$
where $V(x)=(1+|x|^2)^{-\al/2}$ with $\al\in[0,1)$. Later, Wakasugi \cite{W2} considered the decay property of solutions to the wave equation with space-dependent damping, absorbing nonlinearity as
$$
u_{tt}-\tri u+a(x) u_t+|u|^{p-1}u=0\quad\mbox{in }\Om\times\BR_+,
$$
where
$$
a(x)=a_0(1+|x|^2)^{-\al/2}\quad\mbox{or}\quad a_0(1+|x|^2)^{-\al/2}\le a(x)\le a_1(1+|x|^2)^{-\al/2}
$$
with some constants $a_0,a_1>0$ and $\al\in[0,1)$. The author proved how the amplitude of the damping coefficient, the power of nonlinearity and the decay rate of the initial data at the spatial infinity determine the decay rates of the energy and the $L^2$ norm of the solution. In addition, Yang and Fang \cite{Y1} investigated the global well-posedness and blow-up phenomena for a strongly damped wave equation with time-varying source and singular dissipation as
$$
u_{tt}-\tri u-\tri u_t+V(x)u_t=k(t)|u|^{p-1}u\quad\mbox{in }\Om\times\BR_+,
$$
where $V(x)=|x|^{-\si}$, $\si\in[0,2]$ and $k(t)$ is a non-increasing and nonnegative function. They investigated the local well-posedness and the global existence solution based on the cut-off technique, multiplier method, contraction mapping principle and the modified well method. Meanwhile, the blow-up results of solutions with arbitrarily positive initial energy and the lifespan of the blow-up solutions are derived.

Motivated by the above articles, we are interested in the energy decay rates and blow-up results of the viscoelastic wave equation with polynomial nonlinearity and weak damping. We prove the energy is exponential or polynomial decay for imposing some weak conditions on the kernel function, and extend the range of $q$ to $[1,2)$ when the initial energy satisfies some conditions which is only $1\le q<3/2$ for some papers, such as \cite{A2,M3}. In addition, the energy decay rates are faster those in \cite{W2,Y1}. We also establish the blow-up of solution for the initial energy $E(0)<d$. In fact, our results improve some papers mentioned above.

The rest of this article is organized as follows. In Section 2, we introduce notations, assumptions and recall the global well-posedness of solutions. In Section 3, we discuss the energy decay rates for \eqref{1.1} by using Lyapunov method and potential well method under certain conditions on the kernel function. Finally, Section 4 is devoted to the blow-up of the solution by using Levine's convexity method.


\section{Preliminary and well-posedness}

In this section, we fix notations and introduce some  basic definitions, important lemmas and some function spaces for the statements and proofs of our main results. Throughout this article, we denote the norm of the Lebesgue space $L^p(\Om)$ ($1\le p\le\infty$) by $\|\cdot\|_p$, and the inner product of $L^2(\Om)$ by $(\,\cdot\,,\,\cdot\,)$.


\subsection{Some assumptions and definitions}

To begin with, we first fix some assumptions on the kernel function $f(t)$ in \eqref{1.1} which will be used later.\medskip

{\bf(A1)} The function $f\in C^1(\ov{\BR_+};\BR_+)$  is non-increasing and satisfies
\[
f(0)>0,\quad 1-\int_0^\infty f(s)\,\rd s:=\ell\in(0,1).
\]

{\bf(A2)} Under assumption (A1), the function $f(t)$ further satisfies the ordinary differential inequality
$$
f'(t)\le-\xi(t)f^q(t),\quad t\ge0,
$$
where $\xi\in C^1(\ov{\BR_+};\BR_+)$ is non-increasing and $q\in[1,3/2)$ is a constant.\medskip

Next, we introduce the energy functional corresponding with \eqref{1.1} as
\begin{equation}\label{2.5}
E(t):=\f12\|u_t(t)\|_2^2+\f12\left(1-\int_0^t f(s)\,\rd s\right)\|\nb u(t)\|_2^2+\f12(f\circ\nb u)(t)
-\f1p\int_\Om k|u(t)|^p\,\rd x,
\end{equation}
where
$$
(f\circ\nb u)(t):=\int_0^t f(t-s)\|\nb u (t)-\nb u(s)\|_2^2\,\rd s.
$$
By formally differentiating \eqref{2.5} and employing the original problem \eqref{1.1}, it is not difficult to calculate
\begin{equation}\label{2.6}
E'(t)=\f12(f'\circ\nb u)(t)-\f{f(t)}2\|\nb u (t)\|_2^2-\left\|\f{u_t(t)}{|\cdot|^{\si/2}}\right\|_2^2\le0.
\end{equation}
Then it is readily seen from \eqref{2.6} that
\begin{equation}\label{xx2.6}
E(t)+\int_0^t\left\|\f{u_s(s)}{|\cdot|^{\si/2}}\right\|_2^2\,\rd s\le E(0).
\end{equation}
Moreover, we define
\[
J(w):=\f\ell2\|\nb w\|_2^2-\f1p\int_\Om k|w|^p\,\rd x,\quad I(w):=\ell\|\nb w\|_2^2-\int_\Om k|w|^p\,\rd x,\quad w\in H_0^1(\Om)
\]
and
\begin{align}
N & :=\{w\in H_0^1(\Om)\setminus\{0\}\mid I(w)=0\},\nonumber\\
d & :=\inf_{w\in H_0^1(\Om)\setminus\{0\}}\left\{\sup_{\la>0}J(\la w)\right\}=\inf\limits_{w\in N}J(w),\label{2.10}
\end{align}
where $N$ is called the Nehari manifold and $d$ stands for the depth of the potential well. By the definitions of $E(t)$, $J(w)$ and $I(w)$, we also have
\begin{equation}\label{2.11}
J(w)=\left(\f12-\f1p\right)\ell\|\nb w\|_2^2+\f1p I(w)
\end{equation}
and
\begin{equation}\label{2.12}
E(t)\ge\f12\|u_t(t)\|_2^2+\f12(f\circ\nabla u)(t)+J(u(t))\ge J(u(t)).
\end{equation}

\begin{remark}
From the Gagliardo-Nirenberg multiplicative embedding inequality, it is easy to verify that $J$ and $I$ are continuous functionals on $H_0^1(\Om)$ (see \cite{D2,L2}).
\end{remark}


\subsection{Some lemmas and well-posedness}

In this subsection, we prepare several useful lemmas.

\begin{lemma}[see \cite{A3}]\label{llem1}
Let $r$ be a constant satisfying $2\le r\le 2_*=\f{2n}{n-2}$ for $n\ge3$, Then there is an optimal constant $B_r>0$ depending on $r$ such that
$$
\|w\|_r^r\le B_r\|\nb w\|_2^r,\quad\forall\,w\in H_0^1(\Om).
$$
\end{lemma}

\begin{lemma}[see {\cite[Lemma 2.1]{D1}}]\label{lemma2.1}
For any fixed $w\in H_0^1(\Om)\setminus\{0\},$ there exists a unique $\la_*>0$ such that

{\rm(1)} $\f\rd{\rd\la}J(\la w)|_{\la=\la_*}=0$. More precisely, $J(\la w)$ is increasing for $0<\la<\la_*,$ decreasing for $\la_*<\la<\infty$ and attains its maximum at $\la=\la_*$.

{\rm(2)} $I(\la w)>0$ for $0<\la<\la_*,$ $I(\la w)<0$ for $\la_*<\la<\infty$ and $I(\la_*w)=0$.
\end{lemma}

According to Lemma \ref{lemma2.1}(2), we conclude that the Nehari manifold $N$ is nonempty, so that the depth $d$ defined by \eqref{2.10} makes sense.

\begin{lemma}[see {\cite[Lemma 2.2]{D1}}]\label{lemma2.2}
The potential well depth $d$ defined by \eqref{2.10} is positive and there exists a positive function $u\in N$ such that $J(u)=d$.
\end{lemma}

Next, we define two subsets of $H_0^1(\Om)\times L^2(\Om)$ related to problem \eqref{1.1}. For $t\ge0$, set
\begin{align}
W(t) & :=\{(u(t),u_t(t))\in(H_0^1(\Om)\setminus\{0\})\times(L^2(\Om)\setminus\{0\})\mid E(t)<d,I(u(t))>0\},\nonumber\\
V(t) & :=\{(u(t),u_t(t))\in(H_0^1(\Om)\setminus\{0\})\times(L^2(\Om)\setminus\{0\})\mid E(t)<d,I(u(t))<0\}.\label{eq-def-V}
\end{align}
It is obvious that $W(t)\cap V(t)=\emptyset$. Now we recall the local and global existence results of solutions.

\begin{proposition}[Local existence]\label{rtheorem1}
Let $(u^0,u^1)\in H_0^1(\Om)\times L^2(\Om)$ be given and assumption {\rm(A1)} hold. Then there exists $T>0$ such that problem \eqref{1.1} has a unique local weak solution $u$ on $\overline\Om\times[0,T]$ such that
$$
u\in C^1([0,T];\,L^2(\Om))\cap C([0,T];\,H_0^1(\Om)).
$$
\end{proposition}

One can easily show Proposition \ref{rtheorem1} by the Faedo-Galerkin method and the fixed point theorem. See \cite{H1,Y1} for a detailed proof.

\begin{proposition}[Global existence]\label{theorem1}
Let $(u^0,u^1)\in W(0)$ be given and assumption {\rm(A1)} hold. Then there exists a unique global weak solution $u$ to problem \eqref{1.1} and $(u(t),u_t(t))\in W(t)$ for $0<t<\infty$.
\end{proposition}

Interested readers are referred to \cite{P2} for a detailed proof.


\section{Energy decay estimates}

In the sequel, by $C>0$ we denote generic constants which may change from line to line. In this section, we discuss the energy decay results of the system \eqref{1.1}. The main results of this section are as follows.

\begin{theorem}\label{the1}
Suppose that assumptions {\rm(A1)--(A2)} hold, $(u^0,u^1)\in W(0)$ and
\begin{equation}\label{rr3.1}
E(0)<\min\left\{d,{\f{(p-2)\ell}{2p}\left(\f\ell{2K B_p}\right)^{\f2{p-2}}}\right\},
\end{equation}
{where $d$ was defined by $\eqref{2.10},$ $K:=\|k\|_\infty$ and $B_p$ is the constant introduced in Lemma $\ref{llem1}$ with $r=p$. Then for any fixed $t_1>0,$ there exist a constant $C>0$ such that for any $t\ge t_1,$} the energy $E(t)$ satisfy
\begin{equation}\label{3.1}
E(t)\le\left\{\begin{alignedat}{2}
& C E(0)\exp\left(-C\int_{t_1}^t\xi(s)\,\rd s\right), & \quad & q=1,\\
& {C E(0)\left(1+\int_{t_1}^t\xi^{2q-1}(s)\,\rd s\right)^{-\f1{2q-2}}}, & \quad & 1<q<\f32.
\end{alignedat}\right.
\end{equation}
Moreover, if $1<q<\f32$ and
\begin{equation}\label{xx3.1}
\int_0^\infty\left(1+\int_0^t\xi^{2q-1}(s)\,\rd s\right)^{-\f1{2q-2}}\rd t<\infty,\end{equation}
then we have
\begin{equation}\label{x3.1}
E(t)\le C E(0)\left(1+\int_{t_1}^t\xi^q(s)\,\rd s\right)^{-\f1{q-1}}.
\end{equation}
\end{theorem}

\begin{remark}
In fact, the choice of $q$ in inequality \eqref{x3.1} can be extended to $1<q<2$.
\end{remark}

In order to prove Theorem \ref{the1}, we start with introducing a useful lemma.

\begin{lemma}[Hardy-Sobolev inequality, see \cite{Y1,W2}]\label{llem2}
Let $\BR^n=\BR^k\times\BR^{n-k},$ $2\le k\le n$ and $x=(y,z)\in\BR^n=\BR^k\times\BR^{n-k}$. For given $d,\si$ satisfying
$$
1<d<n,\quad0\le\si\le d,\quad\si<k,\quad m=m(\si,n,d):=\f{d(n-\si)}{n-d},
$$
there exists a constant $H=H(\si,n,d,k)>0$ such that for any $u\in W_0^{1,n}(\BR^n),$ there holds
\[
\int_{\BR^n}\f{|u(x)|^m}{|y|^\si}\,\rd x\le H\left(\int_{\BR^n}|\nb u|^d\,\rd x\right)^{\f{n-\si}{n-d}}.
\]
\end{lemma}

Especially, in the case of $m=d=\si$, the above inequality reduces to the classical Hardy inequality. Meanwhile, if $m=d=\sigma=2$ it follows immediately from Lemma \ref{llem2} that
\[
\left\|\f u{|\cdot|^{\si/2}}\right\|_2^2\le H\|\nb u\|_2^2,\quad\forall\,u\in H_0^1(\BR^n).
\]

Next, we introduce some auxiliary functions and lemmas. We define
\begin{align}
\phi(t) & :=\int_\Om u_t(t)u(t)\,\rd x,\label{3.7}\\
\psi(t) & :=-\int_\Om u_t(t)\int_0^t f(t-s)(u(t)-u(s))\,\rd s\rd x.\label{3.8}
\end{align}
We give estimates for $\phi(t)$ and $\psi(t)$ in the next two lemmas.

\begin{lemma}\label{lemma4}
Let $(u^0,u^1)\in W(0)$ be given and assumption {\rm(A1)} hold. Then the function $\phi(t)$ defined by \eqref{3.7} satisfies the following estimate
\begin{equation}\label{3.9}
\phi'(t)\le\|u_t(t)\|_2^2-\f\ell2\|\nb u(t)\|_2^2+\f{1-\ell}\ell(f\circ\nb u)(t)+\f H\ell\left\|\f{u_t(t)}{|\cdot|^{\si/2}}\right\|_2^2+\int_\Om k|u(t)|^p\,\rd x.
\end{equation}
\end{lemma}

\begin{proof}
By the governing equation and the homogeneous boundary condition in \eqref{1.1}, we utilize the divergence theorem to calculate
\begin{align}
\phi'(t) & =\int_\Om\left\{u_{tt}(t)u(t)+u_t^2(t)\right\}\rd x\nonumber\\
& =\|u_t(t)\|_2^2-\|\nb u(t)\|_2^2+\int_\Om\nb u(t)\cdot\int_0^t f(t-s)\nb u(s)\,\rd s\rd x-\int_\Om a u_t(t)u(t)\,\rd x\nonumber\\
& \quad\;\!+\int_\Om k|u(t)|^p\,\rd x\nonumber\\
& \le\|u_t(t)\|_2^2-\ell\|\nb u(t)\|_2^2+\int_\Om\nb u(t)\cdot\int_0^t f(t-s)(\nb u(s)-\nb u(t))\,\rd s\rd x\nonumber\\
& \quad\;\!-\int_\Om a u_t(t)u(t)\,\rd x+\int_\Om k|u(t)|^p\,\rd x.\label{3.10}
\end{align}
Applying {the Cauchy-Schwarz and H\"older's inequalities} to deal with the third term in \eqref{3.10} yields
\begin{align}
& \quad\;\!\int_\Om\nb u(t)\cdot\int_0^t f(t-s)(\nb u(s)-\nb u(t))\,\rd s\rd x\nonumber\\
& \le\f\ell4\|\nb u(t)\|_2^2+\f1\ell\int_\Om\left|\int_0^t\sqrt{f(t-s)}\sqrt{f(t-s)}\,(\nb u(s)-\nb u(t))\right|^2\rd s\rd x\nonumber\\
& \le\f\ell4\|\nb u(t)\|_2^2+\f1\ell\int_0^t{f(t-s)}\,\rd s\int_\Om\int_0^t f(t-s)|\nb u(s)-\nb u(t)|^2\,\rd s\rd x\nonumber\\
& \le\f\ell4\|\nb u(t)\|_2^2+\f{1-\ell}\ell(f\circ\nb u)(t).\label{eq-keyest}
\end{align}
In the same manner and using Lemma \ref{llem2}, we obtain
\begin{equation}\label{3.12}
\int_\Om a u_t(t)u(t)\,\rd x=\int_\Om\f{u_t(t)u(t)}{|x|^\si}\,\rd x\le\f\ell4\|\nb u(t)\|_2^2+\f H\ell\left\|\f{u_t(t)}{|\cdot|^{\si/2}}\right\|_2^2.
\end{equation}
Combining \eqref{3.10}--\eqref{3.12}, we arrive at \eqref{3.9}.
\end{proof}

\begin{lemma}\label{lemma5}
Let $(u^0,u^1)\in W(0)$ be given and assumption {\rm(A1)} hold. Then the function $\psi(t)$ defined by \eqref{3.8} satisfies the estimate
\begin{align*}
\psi'(t) & \le\de\|\nb u(t)\|_2^2+\de\left\|\f{u_t(t)}{|\cdot|^{\si/2}}\right\|_2^2-\left(\int_0^t f(s)\,\rd s-\de\right)\|u_t(t)\|_2^2\\
& \quad\;\!+\de\int_\Om k|u(t)|^p\,\rd x-\f{B_2f(0)}{4\de}(f'\circ\nb u)(t)+C(\de)(f\circ\nb u)(t),
\end{align*}
where $B_2>0$ is the optimal constant in Lemma $\ref{llem1}$ with $r=2,$ $\de>0$ is an arbitrary constant and $C(\de)>0$ is a constant depending on $\de$.
\end{lemma}

\begin{proof}
Similarly to the proof of Lemma \ref{lemma4}, we differentiate $\psi(t)$ by its definition and exploit problem \eqref{1.1} to calculate
\begin{align*}
\psi'(t) & =-\int_\Om u_{tt}(t)\int_0^t f(t-s)(u(t)-u(s))\,\rd s\rd x\\
& \quad\;\!-\int_\Om u_t(t)\int_0^t f'(t-s)(u(t)-u(s))\,\rd s\rd x-\int_0^t f(s)\,\rd s\int_\Om|u_t(t)|^2\,\rd x\\
& =\int_\Om\left\{-\tri u(t)+\int_0^t f(t-s)\tri u(s)\,\rd s+a u_t(t)-k|u(t)|^{p-2}u(t)\right\}\\
& \quad\;\!\times\int_0^t f(t-s)(u(t)-u(s))\,\rd s\rd x+I_5(t)+I_6(t)=\sum_{i=1}^6I_i(t),
\end{align*}
where
\begin{align*}
I_1(t) & :=\left(1-\int_0^t f(s)\,\rd s\right)\int_\Om\nb u(t)\cdot\int_0^t f(t-s)(\nb u(t)-\nb u(s))\,\rd s\rd x,\\
I_2(t) & :=\int_\Om\left|\int_0^t f(t-s)(\nb u(t)-\nb u(s))\,\rd s\right|^2\rd x,\\
I_3(t) & :=\int_\Om a u_t(t)\int_0^t f(t-s)(u(t)- u(s))\,\rd s\rd x,\\
I_4(t) & :=-\int_\Om k|u(t)|^{p-2}u(t)\int_0^t f(t-s)(u(t)- u(s))\,\rd s\rd x,\\
I_5(t) & :=-\int_\Om u_t(t)\int_0^t f'(t-s)(u(t)-u(s))\,\rd s\rd x,\\
I_6(t) & :=-\int_0^t f(s)\,\rd s\int_\Om|u_t(t)|^2\,\rd x.
\end{align*}
Now we estimate each of the above terms. First, similarly to the argument for \eqref{eq-keyest}, we estimate $I_2$ as
\begin{align}
|I_2(t)| & =\int_\Om\left|\int_0^t\sqrt{f(t-s)}\sqrt{f(t-s)}\,(\nb u(t)-\nb u(s))\,\rd s\right|^2\rd x\nonumber\\
& \le\int_0^t f(s)\,\rd s\int_\Om\int_0^t f(t-s)|\nb u(t)-\nb u(s)|^2\,\rd s\rd x\le(1-\ell)(f\circ\nb u)(t).\label{in1}
\end{align}
For $I_1$, recalling assumption (A1) for $f$, we combine the Cauchy-Schwarz inequality, Cauchy's inequality with a constant $\de>0$ and \eqref{in1} to derive
\begin{align*}
|I_1(t)| & \le\int_\Om\left|\nb u(t)\cdot\int_0^t f(t-s)(\nb u(t)-\nb u(s))\,\rd s\right|\rd x\\
& \le\int_\Om|\nb u(t)|\left|\int_0^t f(t-s)(\nb u(t)-\nb u(s))\,\rd s\right|\rd x\\
& \le\de\int_\Om|\nb u(t)|^2\,\rd x+\f1{4\de}\int_\Om\left|\int_0^t f(t-s)(\nb u(t)-\nb u(s))\,\rd s\right|^2\,\rd x\\
& \le\de\|\nb u(t)\|_2^2+\f1{4\de}(1-\ell)(f\circ\nb u)(t).
\end{align*}
Analogously, we apply the Hardy-Sobolev inequality in Lemma \ref{llem2} and \eqref{in1} to estimate
\[
|I_3(t)|\le\de\left\|\f{u_t(t)}{|\cdot|^{\si/2}}\right\|_2^2+\f{H(1-\ell)}{4\de}(f\circ\nb u)(t).
\]
In the same manner, {we estimate $I_5$ as
\begin{align*}
|I_5(t)| & \le\de\int_\Om u_t(t)^2\,\rd x+\f1{4\de}\int_\Om\left(\int_0^t f'(t-s)(u(t)-u(s))\,\rd s\right)^2\rd x\\
& \le\de\|u_t(t)\|_2^2+\f{B_2}{4\de}\int_\Om\left|\int_0^t f'(t-s)(\nb u(t)-\nb u(s))\,\rd s\right|^2\rd x\\
& \le\de\|u_t(t)\|_2^2+\f{B_2}{4\de}\int_0^t f'(t-s)\,\rd s\int_\Om\int_0^t f'(t-s)|\nb u(t)-\nb u(s)|^2\,\rd s\rd x\\
& \le\de\|u_t(t)\|_2^2-\f{B_2f(0)}{4\de}(f'\circ\nb u)(t),
\end{align*}
where we used Lemma $\ref{llem1}$ with $r=2$ and assumption (A1) for $f$. Finally for $I_4$, we utilize Young's inequality with $\de$ the definition of $k$ and \eqref{2.11}--\eqref{2.12} to deduce
\begin{align*}
|I_4(t)| & \le\de\int_\Om k|u(t)|^p\,\rd x+c(\de)\int_\Om k\left(\int_0^t f(t-s)|u(t)-u(s)|\,\rd s\right)^p\rd x\\
& \le\de\int_\Om k|u(t)|^p\,\rd x+(1-\ell)^{p-1}c(\de)\int_\Om k\int_0^t f(t-s)|  u(t)- u(s)|^p\,\rd s\rd x\\
& \le\de\int_\Om k|u(t)|^p\,\rd x+K(1-\ell)^{p-1}c(\de)\int_0^t f(t-s)\|\nb u(t)-\nb u(s)\|_2^p\,\rd s\\
& \le\de\int_\Om k|u(t)|^p\,\rd x+K(1-\ell)^{p-1}c(\de)\left(\f{2pd}{(p-2)\ell}\right)^{\f{p-2}2}(f\circ\nb u)(t),
\end{align*}
where $c(\de)>0$ is a constant depending on $\de$. Here we dealt with $\|\nb u(t)-\nb u(s)\|_2^p$ as}
\begin{align*}
\|\nb u(t)-\nb u(s)\|_2^p & =\|\nb u(t)-\nb u (s)\|_2^{p-2}\|\nb u(t)-\nb u(s)\|_2^2\\
& \le2\left(\f{2pd}{(p-2)\ell}\right)^{\f{p-2}2}\|\nb u(t)-\nb u(s)\|_2^2
\end{align*}
and recall $K=\|k\|_\infty$. Summing up all the above estimates and collecting all the coefficients in front of $(f\circ\nb u)(t)$ as a single constant $C(\de)>0$, we complete the proof of Lemma \ref{lemma5}.
\end{proof}

\begin{lemma}\label{lemma6}
{Let the same assumptions of Theorem $\ref{the1}$ be satisfied. Then for any fixed $t_1>0,$ there exist constants $\ve_1,\ve_2,\al,C>0$ such that for any $t\ge t_1,$} the function
\[
L(t):=E(t)+\ve_1\phi(t)+\ve_2\psi(t)
\]
satisfies
\begin{align}
L(t) & \sim E(t),\label{3.25}\\
L'(t) & \le-\al E(t)+C(f\circ\nb u)(t)\label{3.26}
\end{align}
for all $t\ge t_1$.
\end{lemma}

\begin{proof}
Using Young's and H\"older's inequalities, one can easily get \eqref{3.25} for any sufficiently small $\ve_1,\ve_2>0$ (see \cite{A2} for a detailed proof).

{For \eqref{3.26}, we fix any $t_1>0$ and let $\ve_1,\ve_2,\al>0$ be constants to be selected later. By \eqref{2.5}--\eqref{2.6} and the positivity of $f$, first it is readily seen that
\begin{align*}
E(t) & \le\f12\left(\|u_t(t)\|_2^2+\|\nb u(t)\|_2^2+(f\circ\nb u)(t)\right)-\f1p\int_\Om k|u(t)|^p\,\rd x,\\
E'(t) & \le\f12(f'\circ\nb u)(t)-\left\|\f{u_t(t)}{|\cdot|^{\si/2}}\right\|_2^2.
\end{align*}
Then we combine the above estimates with Lemmas \ref{lemma4}--\ref{lemma5} to deduce for $t\ge t_1$ that
\begin{align*}
L'(t) & \le-\al E(t)+E'(t)+\al E(t)+\ve_1\phi'(t)+\ve_2\psi'(t)\\
& \le-\al E(t)+\left(\f12-\f{\ve_2B_2f(0)}{4\de}\right)(f'\circ\nb u)(t)-\left(1-\f{\ve_1H}\ell-\ve_2\de\right)\left\|\f{u_t(t)}{|\cdot|^{\si/2}}\right\|_2^2\\
& \quad\;\!-\left(\ve_2(f_1-\de)-\ve_1-\f\al2\right)\|u_t(t)\|_2^2-\left(\f{\ve_1\ell}2-\ve_2\de-\f\al2\right)\|\nb u(t)\|_2^2\\
& \quad\;\!+\left(\f\al2+\f{\ve_1(1-\ell)}\ell+\ve_2C(\de)\right)(f\circ\nb u)(t)+\left(\ve_1+\ve_2\de-\f\al p\right)\int_\Om k|u(t)|^p\,\rd x,
\end{align*}
where $f_1:=\int_0^{t_1}f(s)\,\rd s$. For the last term above, we utilize Lemma \ref{llem1} with $r=p$} and \eqref{2.12} to derive
\begin{align*}
\int_\Om k|u(t)|^p\,\rd x\le K\|u(t)\|_p^p\le K B_p\|\nb u(t)\|_2^p\le K B_p\left(\f{2p E(0)}{(p-2)\ell}\right)^{\f{p-2}2}\|\nb u(t)\|_2^2=\xi_1\|\nb u(t)\|_2^2,
\end{align*}
where
\[
\xi_1:=K B_p\left(\f{2p E(0)}{(p-2)\ell}\right)^{\f{p-2}2}.
\]
Thanks to the key assumption \eqref{rr3.1}, we note that $\ell/2>\xi_1$. Then we can dominate $L'(t)$ as
\begin{align}
L'(t) & \le-\al E(t)+\left(\f\al2+\f{\ve_1(1-\ell)}\ell+\ve_2C(\de)\right)(f\circ\nb u)(t)\nonumber\\
& \quad\;\!+\left(\f12-\f{\ve_2B_2f(0)}{4\de}\right)(f'\circ\nb u)(t)-\left(1-\f{\ve_1H}\ell-\ve_2\de\right)\left\|\f{u_t(t)}{|\cdot|^{\si/2}}\right\|_2^2\nonumber\\
& \quad\;\!-\left(\ve_2(f_1-\de)-\ve_1-\f\al2\right)\|u_t(t)\|_2^2-\left\{\ve_1\left(\f\ell2-\xi_1\right)-\ve_2\de(1+\xi_1)-\f\al2\right\}\|\nb u(t)\|_2^2.\label{est-L'}
\end{align}

In comparison with the desired inequality \eqref{3.26}, it suffices to choose constants $\de,\ve_1,\ve_2,\al>0$ suitably such that the last 4 terms on the right-hand side of \eqref{est-L'} are negative, that is,
\begin{gather}
\f12-\f{\ve_2B_2f(0)}{4\de}\ge0,\quad1-\f{\ve_1H}\ell-\ve_2\de\ge0,\label{eq-pos1}\\
\ve_2(f_1-\de)-\ve_1-\f\al2\ge0,\quad\ve_1\left(\f\ell2-\xi_1\right)-\ve_2\de(1+\xi_1)-\f\al2\ge0.\label{eq-pos2}
\end{gather}
Owing to $\ell/2>\xi_1$, we can choose sufficiently small $\de>0$ such that
\[
f_1-\de>\f{\de(1+\xi_1)}{\ell/2-\xi_1}>0.
\]
Then we can restrict $\ve_1,\ve_2>0$ as
\[
\ve_2(f_1-\de)>\ve_1>\ve_2\f{\de(1+\xi_1)}{\ell/2-\xi_1}.
\]
Next, we select first $\ve_2>0$ and then $\ve_1>0$ as
\[
\ve_2<\min\left\{\f{2\de}{B_2f(0)},\f1\de\right\},\quad\ve_1\le\f{\ell(1-\ve_2\de)}H,
\]
such that \eqref{eq-pos1} is achieved. Finally, we can achieve \eqref{eq-pos2} by choosing sufficiently small $\al>0$. Hence, we complete the proof of \eqref{3.26} by collecting the coefficient of $(f\circ\nb u)(t)$ in \eqref{est-L'} into a single constant $C>0$.
\end{proof}

{In order to treat $(f\circ\nb u)(t)$, we need an additional lemma.}

\begin{lemma}[see \cite{A2}]\label{lemm1}
Under assumptions {\rm(A1)--(A2),} the solution $u$ to the problem \eqref{1.1} satisfies
\[
\xi(t)(f\circ\nb u)(t)\le C(-E'(t))^{\f1{2q-1}},\quad t>0.
\]
\end{lemma}

Now, we are in the position to tackle Theorem \ref{the1}.

\begin{proof}[\bf Proof of Theorem \ref{the1}]
Multiplying both sides of \eqref{3.26} by $\xi(t)$ and applying Lemma \ref{lemm1}, we see
\[
\xi(t)L'(t)\le -\al \xi(t)E(t)+C(-E'(t))^{\f1{2q-1}}.
\]
Further multiplying both sides of of the above inequality by $(\xi(t)E(t))^\ga$ with $\ga:=2q-2$, we employ Young's inequality to estimate
\begin{align}
\xi^{\ga+1}(t)E^\ga(t)L'(t) & \le-\al(\xi(t)E(t))^{\ga+1}+C(\xi(t)E(t))^\ga(-E'(t))^{\f1{\ga+1}}\nonumber\\
& \le-\f\al2(\xi(t)E(t))^{\ga+1}-CE'(t).\label{3.31}
\end{align}
Set $F(t):=(\xi(t)E(t))^{\ga+1}L(t)+C E(t)$. Using the non-increasing of $\xi(t),E(t)$and \eqref{3.25}, we have the equivalence $F(t)\sim E(t)$. Then we differentiate $F(t)$ and apply \eqref{3.31} to deduce
\begin{align}
F'(t)\le(\xi(t)E(t))^{\ga+1}L'(t)+C E'(t)\le-\f\al2(\xi(t)E(t))^{\ga+1}\le-C(\xi(t)F(t))^{2q-1},\label{3.32}
\end{align}
which is an ordinary differential inequality with respect to $F$. Below we discuss the cases of $q=1$ and otherwise separately.

For $q=1$, we immediately have $(\log|F(s)|)'\le-C\xi(s)$ for $s\ge t_1$. Then integrating both sides of this inequality from $t_1$ to any $t>t_1$ yields
\[
F(t)\le F(t_1)\exp\left(-C\int_{t_1}^t\xi(s)\,\rd s\right).
\]
Again by the equivalence $F(t)\sim E(t)$ and the monotonicity of $E(t)$, we conclude
\[
E(t)\le C F(t)\le C E(t_1)\exp\left(-C\int_{t_1}^t\xi(s)\,\rd s\right)\le C E(0)\exp\left(-C\int_{t_1}^t\xi(s)\,\rd s\right).
\]
For $1<q<3/2$, multiplying both sides of \eqref{3.32} by $F^{1-2q}(t)$ and integrating over $[t_1,t]$, we can analogously verify the second inequality of \eqref{3.1} with the aid of the equivalence $F(t)\sim E(t)$ and the monotonicity of $E(t)$.

To show \eqref{x3.1}, first it follows from simple calculation based on \eqref{3.1} and \eqref{xx3.1} that
\[
\int_0^\infty E(t)\,\rd t<\infty.
\]
Then we can estimate the following quantity as
\begin{align}
\La(t) & :=\int_0^t\|\nb u(t)-\nb u(s)\|_2^2\,\rd s\le C\int_0^t\left(\|\nb u(t)\|_2^2+\|\nb u(s)\|_2^2\right)\rd s\nonumber\\
& \le C\int_0^t(E(t)+E(s))\,\rd s\le 2C\int_0^t E(s)\,\rd s<2C\int_0^\infty E(s)\,\rd s<\infty.\label{3.50}
\end{align}
Without loss of generality, we can assume $\La(t)>0$ for all $t\ge t_1$, since otherwise \eqref{1.1} yields an exponential decay. Then we take advantage of H\"older's inequality, the assumption (A2), \eqref{3.50}, \eqref{2.6} and the monotonicity of $\xi$ to derive
\begin{align*}
& \quad\;\!\xi(t)(f\circ\nb u)(t)\le\int_0^t\xi(t-s)f(t-s)\|\nb u(t)-\nb u(s)\|_2^2\,\rd s\\
& \le\left(\int_0^t\xi(t-s)\|\nb u(t)-\nb u(s)\|_2^2\,\rd s\right)^{1-1/q}\left(\int_0^t\xi(t-s)f^q(t-s)\|\nb u(t)-\nb u(s)\|_2^2\,\rd s\right)^{1/q}\\
& \le\left(\xi(0)\int_0^t\|\nb u(t)-\nb u(s)\|_2^2\,\rd s\right)^{1-1/q}(-(f'\circ\nb u)(t))^{1/q}\\
& \le(\xi(0)\La(t))^{1-1/q}(-2E'(t))^{1/q}\le C(-E'(t))^{\f1q}.
\end{align*}
Similarly to the argument at the beginning of this proof, we multiply both sides of \eqref{3.26} by $\xi(t)$ and employ the above inequality to deduce
\[
\xi(t)L'(t)\le-\al\xi(t)E(t)+C(-E'(t))^{1/q}.
\]
Further multiplying both sides of the above inequality by $(\xi(t)E(t))^\be$ with $\be:=q-1$, we employ Young's inequality with $\ve>0$ to estimate
\begin{align*}
\xi^{\be+1}(t)E^\be(t)L'(t) & \le-\al(\xi(t)E(t))^{\be+1}+C(\xi(t)E(t))^\be(-E'(t))^{\f1{\be+1}}\\
& \le-\al(\xi(t)E(t))^{\be+1}+\ve(\xi(t)E(t))^{\be+1}-C_\ve E'(t).
\end{align*}
Setting $\wt F(t):=\xi^{\be+1}(t)E^\be(t)L(t)+C_\ve E(t)$, we can choose $\ve>0$ sufficiently small such that $\wt F(t)\sim E(t)$. Then there exist a constant $C>0$ such that
\[
\wt F'(t)\le-C\xi^{\be+1}(t)\wt F^{\be+1}(t).
\]
Eventually, we can conclude \eqref{x3.1} by repeating the previous argument.
\end{proof}


\section{Finite time blow-up}

In this section, we discuss the finite time blow-up phenomena of the solution to problem \eqref{1.1} by virtue of Levine's convexity method. Recalling the definition \eqref{eq-def-V} of $V(t)$, we start with describing the main results of this part as follows.

\begin{theorem}\label{theorem4.1}
Let assumption {\rm(A1)} hold, $(u^0,u^1)\in V(0)$ and $E(0)=\te d$ with $\te\le1,$ where $d$ was defined by \eqref{2.10}. Assume that the kernel function $f$ satisfies
\begin{equation}\label{5.1}
\int_0^\infty f(s)\,\rd s<\f{p-2}{p-2+((1-\te_+)^2p+2\te_+(1-\te_+))^{-1}}
\end{equation}
with $\te_+:=\max\{0,\te\}$. In the case of $E(0)=0,$ further assume that the inner product $(u^0,u^1)>0$. Then the solution to problem \eqref{1.1} blows up in finite time. Moreover, the blow-up time $T_\infty$ can be estimated from above as
\begin{equation}\label{eq-est-T1}
T_\infty\le\f{2\|u^0\|_2^2+2\eta\mu^2}{(p-2)(u^0,u^1)+\eta\mu-2\|u^0|\cdot|^{-\si/2}\|_2^2},
\end{equation}
where $\eta,\mu$ are constants to be specified later.
\end{theorem}

\begin{theorem}\label{theorem4.2}
Let assumption {\rm(A1)} hold and $u$ be a weak solution to problem \eqref{1.1} that blows up in finite time. Then the blow-up time $T_\infty$ can be estimated from below as
$$
T_\infty\ge\f1{(p-2)K B_{2(p-1)}}\ln\left\{1+\ell^{p-1}\left(\|u^1\|_2^2+\|\nb u^0\|_2^2\right)^{2-p}\right\},
$$
where we recall that $K=\| k\|_\infty$ and $B_{2(p-1)}$ is the optimal constant in Lemma $\ref{llem1}$ with $r=2(p-1)$.
\end{theorem}


\subsection{Proof of Theorem \ref{theorem4.1}}

To prove Theorem \ref{theorem4.1}, we first introduce some useful lemmas.

\begin{lemma}[see \cite{L3}]\label{lem4.1} 
Let $G(t)$ be a positive $C^2$ function satisfying
$$
G(t)G''(t)-(1+\rho)(G'(t))^2\ge 0,\quad\forall\,t>0
$$
with some constants $\rho>0$. If $G(0)>0$ and $G'(0)>0,$ then there exists a constant $T^*\le \f{G(0)}{\rho G'(0)}$ such that $\lim_{t\to T^*}G(t)=\infty$.
\end{lemma}

\begin{lemma}[see \cite{H1}]\label{lem4.2}
If $(u^0,u^1)\in V(0),$ then $(u(t),u_t(t))\in V(t)$ for $t\in(0,T)$.
\end{lemma}

We prove Theorem \ref{theorem4.1} by contradiction, that is, we suppose that the solution $u$ is global in time. For any $T>0$, we define
\begin{equation}\label{5.2}
G(t):=\|u(t)\|_2^2+\int_0^t\left\|\f{u(s)}{|\cdot|^{\si/2}}\right\|_2^2\,\rd s+(T-t)\left\|\f{u^0}{|\cdot|^{\si/2}}\right\|_2^2+\eta(t+\mu)^2,\quad t\in[0,T],
\end{equation}
where $\mu>0$ and $\eta\ge0$ are constants to be specified later. Then it is clear that $G(t)>0$ for $t\in[0,T]$.

We compute the first-order differential and second-order differential of \eqref{5.2}, respectively, as follows
\begin{equation}\label{5.3}
G'(t)=2\int_\Om u(t)u_t(t)\,\rd x+2\int_0^t\!\!\!\int_\Om\f{u(s)u_s(s)}{|x|^\si}\,\rd x\,\rd s+2\eta(t+\mu),
\end{equation}
and
\begin{align}
G''(t) & =2\|u_t(t)\|_2^2+2\int_\Om u(t)u_{tt}(t)\,\rd x+2\int_\Om\f{u(t)u_t(t)}{|x|^\si}\,\rd x+2\eta\nonumber\\
& =2\|u_t(t)\|_2^2-2\left(1-\int_0^t f(s)\,\rd s\right)\|\nb u(t)\|_2^2+2\int_\Om\nb u(t)\cdot\int_0^t f(t-s)(\nb u(s)-\nb u(t))\,\rd s\rd x\nonumber\\
& \quad\;\!+2\int_\Om k|u(t)|^p\,\rd x+2\eta\nonumber\\
& =-2p E(t)+(p+2)\|u_t(t)\|_2^2+(p-2)\left(1-\int_0^t f(s)\,\rd s\right)\|\nb u(t)\|_2^2+p(f\circ\nb u)(t)\nonumber\\
& \quad\;\!+2\int_\Om\nb u(t)\cdot\int_0^t f(t-s)(\nb u(s)-\nb u(t))\,\rd s\rd x+2\eta.
\end{align}
Applying Cauchy-Schwarz inequality and exploiting \eqref{5.3}, we estimate
\begin{align}
\f{{G'(t)}^2}4 & \le\left(\|u(t)\|_2^2+\int_0^t\left\|\f{u(s)}{|\cdot|^{\si/2}}\right\|_2^2\,\rd s+\eta(t+\mu)^2\right)\left(\|u_t(t)\|_2^2+\int_0^t\left\|\f{u_s(s)}{|\cdot|^{\si/2}}\right\|_2^2\,\rd s+\eta\right)\nonumber\\
& =\left(G(t)-(T-t)\left\|\f{u^0}{|\cdot|^{\si/2}}\right\|_2^2\right)\left(\|u_t(t)\|_2^2+\int_0^t\left\|\f{u_s(s)}{|\cdot|^{\si/2}}\right\|_2^2\,\rd s+\eta\right)\nonumber\\
& \le G(t)\left(\|u_t(t)\|_2^2+\int_0^t\left\|\f{u_s(s)}{|\cdot|^{\si/2}}\right\|_2^2\,\rd s+\eta\right).\label{5.5}
\end{align}
Combining \eqref{5.2}--\eqref{5.5}, we obtain
\begin{align*}
G(t)G''(t)-\f{p+2}4{G'(t)}^2 & \ge G(t)\left\{G''(t)-(p+2)\left(\|u_t(t)\|_2^2+\int_0^t\left\|\f{u_s(s)}{|\cdot|^{\si/2}}\right\|_2^2\,\rd s+\eta\right)\right\}\\
& =G(t)\Bigg\{-2p E(t)+(p-2)\left(1-\int_0^t f(s)\,\rd s\right)\|\nb u(t)\|_2^2+p(f\circ\nb u)(t)\\
& \qquad\qquad\:\!+2\int_\Om\nb u(t)\cdot\int_0^t f(t-s)(\nb u(s)-\nb u(t))\,\rd s\rd x\\
& \qquad\qquad\left.-(p+2)\int_0^t\left\|\f{u_s(s)}{|\cdot|^{\si/2}}\right\|_2^2\,\rd s-p\eta\right\}.
\end{align*}
Utilizing Young's inequality with $\ve>0$ and \eqref{xx2.6}, one has
\begin{align}
G(t)G''(t)-\f{p+2}4{G'(t)}^2 & \ge G(t)\Bigg\{-2p E(0)+\left((p-2)-(p-2+\ve)\int_0^t f(s)\,\rd s\right)\|\nb u(t)\|_2^2\nonumber\\
& \qquad\qquad\left.+\left(p-\f1\ve\right)(f\circ\nb u)(t)+(p-2)\int_0^t\left\|\f{u_s(s)}{|\cdot|^{\si/2}}\right\|_2^2\,\rd s-p\eta\right\}\nonumber\\
& =:G(t)\ze(t).\label{5.6}
\end{align}
To demonstrate the positivity of $\ze(t)$, we divide the proof into three cases depending on the initial energy $E(0)$, i.e., $E(0)<0$, $E(0)=0$ and $0<E(0)<d$.\medskip

{\bf Case 1 }  $E(0)<0$, i.e., $\te<0$. 

Taking $\ve=1/p$ in \eqref{5.6} and choosing $0<\eta<-2E(0)$, then it follows from \eqref{5.1} that
\begin{align*}
\ze(t) & =-2p E(0)+\left((p-2)-\left\{p-2+\f1p\right)\int_0^t f(s)\,\rd s\right\}\|\nb u(t)\|_2^2\nonumber\\
& \quad\;\!+(p-2)\int_0^t\left\|\f{u_s(s)}{|\cdot|^{\si/2}}\right\|_2^2\,\rd s-p\eta\ge0.
\end{align*}

{\bf Case 2 }  $E(0)=0$, i.e., $\te=0$.

Taking $\ve=1/p$ in \eqref{5.6} and choosing $\eta=0$, we see from \eqref{5.1} that
\[
\ze(t)=\left\{(p-2)-\left(p-2+\f1p\right)\int_0^t f(s)\,\rd s\right\}\|\nb u(t)\|_2^2+(p-2)\int_0^t\left\|\f{u_s(s)}{|\cdot|^{\si/2}}\right\|_2^2\,\rd s\ge0.
\]

{\bf Case 3 } $0<E(0)<d$, i.e., $0<\te<1$.

Taking $\ve=((1-\te)p+2\te)^{-1}$ in \eqref{5.6}, we get
\begin{align*}
\ze(t) & =\left\{(p-2)-\left(p-2+\f1{(1-\te)p+2\te}\right)\int_0^t f(s)\,\rd s\right\}\|\nb u(t)\|_2^2-2p E(0)-p\eta
\\
& \quad\;\!+(p-(1-\te)p-2\te)(f\circ\nb u)(t)+(p-2)\int_0^t\left\|\f{u_s(s)}{|\cdot|^{\si/2}}\right\|_2^2\,\rd s.
\end{align*}
Due to the condition \eqref{5.1}, we yield
\[
(p-2)-\left(p-2+\f1{(1-\te)p+2\te}\right)\int_0^t f(s)\,\rd s=\te(p-2)\left(1-\int_0^t f(s)\,\rd s\right)
\]
and thus
\begin{align}
\ze(t) & =-2p E(0)+\te(p-2)\left\{\left(1-\int_0^t f(s)\,\rd s\right)\|\nb u(t)\|_2^2+(f\circ\nb u)(t)\right\}\nonumber\\
& \quad\;\!+(p-2)\int_0^t\left\|\f{u_s(s)}{|\cdot|^{\si/2}}\right\|_2^2\,\rd s-p\eta\nonumber\\
& \ge-2p E(0)+\te(p-2)\left\{\left(1-\int_0^t f(s)\,\rd s\right)\|\nb u(t)\|_2^2+(f\circ\nb u)(t)\right\}-p\eta.\label{5.11}
\end{align}
On the other hand, Lemma \ref{lem4.2} implies that $(u(t),u_t(t))\in V(t)$ for $t\in [0,T]$ and $I(u(t))<0$. Therefore, there exists a constant $\wt\la\in(0,1)$ such that $I(\wt\la u(t))=0$. Hence from the definition of $J(u(t))$ and $d$, we have
\[
d\le J(\wt\la u(t))<\f{p-2}{2p}\left(1-\int_0^t f(s)\,\rd s\right)\|\nb u(t)\|_2^2+\f{p-2}{2p}(f\circ\nb u)(t).
\]
Due to $u$ is continuous on $[0,T]$, there exists a constant $\ga>0$ such that
\begin{equation}\label{5.13}
d+\ga<\f{p-2}{2p}\left(1-\int_0^t f(s)\,\rd s\right)\|\nb u(t)\|_2^2+\f{p-2}{2p}(f\circ\nb u)(t).
\end{equation}
Hence, \eqref{5.11} and \eqref{5.13} imply
\[
\ze(t)\ge-2p\te d+2p\te\f{p-2}{2p}\left\{\left(1-\int_0^t f(s)\,\rd s\right)\|\nb u(t)\|_2^2+(f\circ\nb u)(t)\right\}-p\eta\ge2\te p\ga-p\eta.
\]
Choosing $\eta>0$ sufficiently small such that $2\te p\ga-p\eta>0$, we arrive at $\ze(t)\ge0$ in Case 3.\medskip

Eventually, it turns out that $\ze(t)\ge0$ in all cases, and thus \eqref{5.6} indicates
\[
G(t)G''(t)-\f{p+2}4{G'(t)}^2\ge0.
\]
Next, we shall show that $G'(0)>0$. From \eqref{5.3}, one gets
\[
G'(0)=2(u^0,u^1)+2\eta\mu.
\]
In the case of $E(0)=0$, the condition $(u^0,u^1)>0$ immediately gives $G'(0)>0$. In others cases, we can choose $\mu>0$ sufficiently large such that $G'(0)>0$. Now we are in a position to apply Lemma \ref{lem4.1} to conclude
\begin{equation}\label{4.1}
\lim_{t\to T_\infty}G(t)=\infty,
\end{equation}
where
$$
T_\infty\le\f{4G(0)}{(p-2)G'(0)}=\f{2\|u^0\|_2^2+2T\|u^0|\cdot|^{-\si/2}\|_2^2+2\eta\mu^2}{(p-2)(u^0,u^1)+\eta\mu}.
$$
Thus we arrive at the upper bound \eqref{eq-est-T1} of $T_\infty$ as desired. Finally, combining \eqref{5.2} with \eqref{4.1} yields
$$
\lim_{t\to T_\infty}\left(\|u(t)\|_2^2+\int_0^t\left\|\f{u(s)}{|\cdot|^{\si/2}}\right\|_2^2\,\rd s\right)=\infty,
$$
which contradicts with our assumption. Hence we confirmed that the weak solution $u$ of problem \eqref{1.1} blows up in finite time.


\subsection{Proof of Theorem \ref{theorem4.2}}

We consider $M:[0,T]\rightarrow\mathbb{R}_+$ defined by
\[
M(t):=\|u_t(t)\|_2^2+\left(1-\int_0^t f(s)\,\rd s\right)\|\nb u(t)\|_2^2+(f\circ\nb u)(t).
\]
Differentiating $M(t)$ and employing the original problem \eqref{1.1}, we calculate
\begin{align}
M'(t) & =2(u_t(t),u_{tt}(t))+2\left(1-\int_0^t f(s)\,\rd s\right)(\nb u(t),\nb u_t(t))+(f'\circ\nb u)(t)-f(t)\|\nb u(t)\|_2^2\nonumber\\
& \quad\;\!-2\int_\Om\nb u_t(t)\cdot\int_0^t f(t-s)(\nb u(s)-\nb u(t))\,\rd s\rd x\nonumber\\
& =-2\int_\Om a |u_t(t)|^2\,\rd x+2\int_\Om k|u(t)|^{p-2}u(t)u_t(t)\,\rd x+(f'\circ\nb u)(t)-f(t)\|\nb u(t)\|_2^2\nonumber\\
& {\le2\int_\Om k|u(t)|^{p-2}u(t)u_t(t)\,\rd x.}\label{x5.2}
\end{align}
Using Young's inequality and {Lemma \ref{llem1}}, one has
\begin{align}
2\int_\Om k|u(t)|^{p-2}u(t)u_t(t)\,\rd x & \le2K\|u_t(t)\|_2 \|u(t)\|_{2(p-1)}^{p-1}\le2K B_{2(p-1)}\|u_t(t)\|_2\|\nb u(t)\|_2^{p-1}\nonumber\\
& \le K B_{2(p-1)}\|u_t(t)\|_2^2+K B_{2(p-1)}\|\nb u(t)\|_2^{2(p-1)}\nonumber\\
& \le K B_{2(p-1)}\left(M(t)+\f{M^{p-1}(t)}{\ell^{p-1}}\right).\label{x5.3}
\end{align}
Combining \eqref{x5.2} and \eqref{x5.3}, we obtain
\[
M'(t)\le K B_{2(p-1)}\left(M(t)+\f{M^{p-1}(t)}{\ell^{p-1}}\right).
\]
Multiplying both sides of the above inequality by $(2-p)\,\e^{(p-2)K B_{2(p-1)}t}M^{1-p}(t)$, we get
$$
\left(\e^{(p-2)K B_{2(p-1)}t}M^{2-p}(t)\right)'\ge\f{(2-p)K B_{2(p-1)}}{\ell^{p-1}}\,\e^{(p-2)K B_{2(p-1)}t}.
$$
Integrating the above inequality over $[0,t]$ yields
\[
M^{2-p}(t)\ge\left(M^{2-p}(0)+\f1{\ell^{p-1}}\right)\e^{-(p-2)K B_{2(p-1)}t}-\f1{\ell^{p-1}}.
\]
By $p>2$, it reveals that $M(t)$ remains bounded for $t<T_*$ with
\begin{align*}
T_* & :=\f1{(p-2)KB_{2(p-1)}}\ln\left(1+\ell^{p-1}M^{2-p}(0)\right)\\
& =\f1{(p-2)KB_{2(p-1)}}\ln\left(1+\ell^{p-1}\left(\|u^1\|_2^2+\|\nb u^0\|_2^2\right)^{2-p}\right)
\end{align*}
by recalling the definition of $M(0)$. Therefore, we can conclude that $T_*$ is a lower bound for the blow-up time $T_\infty$, which completes the proof of Theorem \ref{theorem4.2}.


\section*{Acknowledgements}

The first author is supported by China Scholarship Council. The second author is supported by JSPS KAKENHI Grant Numbers JP22K13954, JP23KK0049 and Guangdong Basic and Applied Basic Research Foundation (No.\! 2025A1515012248).


\end{document}